\documentclass[12pt,reqno]{amsart}
\usepackage{enumerate, latexsym, amsmath, amsfonts, amssymb, amsthm, color}
\def\pmod #1{\ ({\rm{mod}}\ #1)}
\def\Z{\Bbb Z}
\def\N{\Bbb N}

\def\l{\left}
\def\r{\right}
\def\bg{\bigg}
\def\({\bg(}
\def\){\bg)}
\def\t{\text}
\def\f{\frac}

\def\mo{{\rm{mod}\ }}

\def\ls{\leqslant}
\def\gs{\geqslant}

\def\bi{\binom}

\def\eq{\equiv}

\def\da{\delta}

\newcommand{\qbinom}[2]{\genfrac{[}{]}{0pt}{}{#1}{#2}_{q}}

\theoremstyle{plain}
\newtheorem{theorem}{Theorem}

\newtheorem{lemma}{Lemma}

\newtheorem{conjecture}{Conjecture}
\theoremstyle{definition}

\theoremstyle{remark}
\newtheorem{remark}{Remark}

 \vspace{4mm}

\begin{document}

\hbox{Preprint, {\tt arXiv:2607.07638}}
\medskip

\title
[{A new kind of numbers and related congruences}]
{A new kind of numbers \\ and related congruences}

\author
[Z.-W. Sun] {Zhi-Wei Sun}

\address{School of Mathematics, Nanjing
University, Nanjing 210093, People's Republic of China}
\email{zwsun@nju.edu.cn}

\keywords{Multi-nomial coefficients, Domb numbers, Franel numbers, $p$-adic congruences.
\newline \indent Supported by the Natural Science Foundation of China (grant no. 12371004).}

\begin{abstract} For integers $l>0$ and $m\geqslant0$, we introduce the numbers
$$S_l^{(m)}(n)=\sum_{k_1,\ldots,k_l\in\mathbb N\atop k_1+\cdots+k_l=n}\binom n{k_1,\ldots,k_l}^m
\ \ (n=0,1,2,\ldots),$$
and prove that for any prime $p$ not dividing $l+1$ we have the congruence
$$\sum_{n=1}^{p-1}\frac{(-1)^{mn}}{n^{m-1}}S_l^{(m)}(n)\equiv0\pmod p.$$
We also obtain a $q$-analogue of this result.

For the Domb numbers given by
$$D(n)=\sum_{k=0}^n\binom nk^2\binom{2k}k\binom{2(n-k)}{n-k}=S_4^{(2)}(n)\ \ (n=0,1,2,\ldots),$$
we confirm a previous conjecture which states that
$$\sum_{n=1}^{p-1}\frac{D(n)}n\equiv\left(\frac p3\right)\frac 25pB_{p-2}\left(\frac13\right)\pmod{p^2}$$
for any prime $p$, where $(\frac p3)$ is the Legendre symbol, and $B_{p-2}(x)$ is the Bernoulli polynomial of degree $p-2$.
\end{abstract}
\maketitle

\section{Introduction}
\setcounter{lemma}{0}
\setcounter{theorem}{0}
\setcounter{corollary}{0}
\setcounter{remark}{0}

Let $m\in\Z^+=\{1,2,3,\ldots\}$. The {\it Franel numbers of order $m$} are defined by
\begin{equation}\label{Fr}f_n^{(m)}:=\sum_{k=0}^n\bi nk^m\ \ (n\in\N=\{0,1,2,\ldots\}).
\end{equation}
Those $f_n=f_n^{(3)}$ with $n\in\N$ are the usual Franel numbers introduced by J. Franel \cite{F} in 1894 who noted the recurrence relation
$$(n+1)^2f_{n+1}=(7n^2+7n+2)f_n+8n^2f_{n-1}\ \ (n=1,2,3,\ldots).$$
Note that
$$f_n^{(2)}=\sum_{k=0}^n\bi nk\bi n{n-k}=\bi{2n}n\quad\t{for any}\ n\in\N$$
by the Chu-Vandermone identity (cf. \cite[p.\,169]{GKP}).
In 2013, the author \cite{S13} proved that for any prime $p>3$ we have
\begin{equation}\label{fp2}\sum_{n=1}^{p-1}\f{(-1)^n}{n}f_n\eq0\pmod{p^2}\end{equation}
and
\begin{equation}\label{f-p}\sum_{n=1}^{p-1}\f{(-1)^{mn}}{n^{m-1}}f_n^{(m)}\eq0\pmod p.
\end{equation}

The {\it Domb numbers} (cf. \cite{OEIS}) are given by
\begin{equation}\label{Domb}D(n)=\sum_{k=0}^n\bi nk^2\bi{2k}k\bi{2(n-k)}{n-k}\quad (n\in\N).
\end{equation}
The values of $D(0),\ldots,D(9)$ are
$$1,\, 4,\, 28,\, 256,\, 2716,\, 31504,\, 387136,\, 4951552,\, 65218204,\, 878536624$$
respectively. It is well known that
$$(n+1)^3D(n+1)=2(2n+1)(5n^2+5n+2)D(n)-64n^3D(n-1)$$
for all $n\in\Z^+$.

The Domb numbers have several combinatorial interpretations. For example, $D(n)$ with $n\in\N$
 is the number of $2n$-step polygons on the diamond lattice, and it is also the $2n$-moment of the distance from the origin of a $4$-step random walk in the plane.

By H. H. Chan, S. H. Chan and Z. Liu \cite{CCL} and M. D. Rogers \cite{Rogers},
\begin{equation*}\sum_{k=0}^\infty\f{5k+1}{64^k}D(k)=\f 8{\sqrt3\,\pi}
\ \t{and}\ \sum_{k=0}^\infty\f{3k+1}{(-32)^k}D(k)=\f2{\pi}.\end{equation*}
Motivated by this, the author \cite{S20} proved that
$$\sum_{k=1}^\infty\f{k^2(5k-3)}{64^k}D(k)=\f{8\sqrt3}{9\pi}
\ \t{and}\ \sum_{k=1}^\infty\f{k(3k+1)^2}{(-32)^k}D(k)=-\f 2{\pi}.$$
See also J.M. Campbell \cite{Cam} for relations between Domb numbers and modular forms,
and J.-C. Liu \cite{Liu} for some supercongruences involving Domb numbers.

In 2011, the author posed the following conjecture (cf.
\cite[Conjecture 5.1]{S13} and \cite[Conjecture 76]{S19}) which remains open.

\begin{conjecture} Let $p>5$ be a prime. Then
\begin{align*}&\sum_{k=0}^{p-1}D(k)
\\\eq&\begin{cases} 4x^2-2p\pmod{p^2}&\t{if}\ p\eq1,4\ (\mo\ 15)\ \t{and}\ p=x^2+15y^2,
\\2p-12x^2\pmod{p^2}&\t{if}\ p\eq2,8\ (\mo\ 15)\ \t{and}\ p=3x^2+5y^2,
\\0\pmod{p^2}&\t{if}\ (\f{-15}{p})=-1,\ \t{i.e.,}\ p\eq 7,11,13,14\pmod{15}.\end{cases}
\end{align*}
where $(\f{\cdot}p)$ is the Legendre symbol, and $x$ and $y$ are integers.
\end{conjecture}

Now we introduce a new kind of numbers which unifies Franel numbers of order $m$ and Domb numbers.

For $l\in\Z^+$ and $m,n\in\N$, we define
\begin{equation}S_l^{(m)}(n):=\sum_{k_1,\ldots,k_l\in\N\atop k_1+\cdots+k_l=n}\bi n{k_1,\ldots,k_l}^m,
\end{equation}
and call $l$ and $m$ its {\it level} and {\it height}, respectively.
Clearly,
$$S_l^{(0)}(n)=\bi{l+n-1}n\ \ \t{and}\ \ S_l^{(1)}(n)=l^n$$
by enumerative combinatorics and multi-nomial theorem. Also,
$$S_2^{(m)}(n)=f_n^{(m)}\quad\t{if}\ m>0.$$

As in \cite{S16}, for any $n\in\N$ we define the polynomial
$$g_n(x):=\sum_{k=0}^n\bi nk^2\bi{2k}kx^k$$
and write $g_n$ for $g_n(1)$. V.J.W. Guo, G.-S. Mao and H. Pan \cite{GMP} proved the author's conjecture that
$$\f1n\sum_{k=0}^{n-1}(4k+3)g_k(x)\in\Z[x]\quad\t{for all}\ n\in\Z^+.$$
Note that
\begin{align*} \sum_{i,j\in\N\atop i+j\ls n}\l(\f{n!}{i!j!(n-i-j)!}\r)^2
&=\sum_{i,j\in\N\atop i+j\ls n}\bi{n}{i+j}^2\bi{i+j}j^2
\\&=\sum_{k=0}^n\bi nk^2\sum_{j=0}^k\bi kj^2=\sum_{k=0}^n\bi nk^2\bi{2k}k.
\end{align*}
and so
\begin{equation}S_3^{(2)}(n)=g_n.\end{equation}
By \cite{S16}, for any odd prime $p$ we have
$$\sum_{n=1}^{p-1}\f{g_n(x)}n\eq0\pmod p,$$
and in particular
\begin{equation}\label{g-p}\sum_{n=1}^{p-1}\f{g_n}n\eq0\pmod p.
\end{equation}

For any $n\in\N$, clearly
\begin{align*} \sum_{i,j\in\N\atop i+j\ls n}\l(\f{n!}{i!j!(n-i-j)!}\r)^3
&=\sum_{i,j\in\N\atop i+j\ls n}\bi{n}{i+j}^3\bi{i+j}j^3
\\&=\sum_{k=0}^n\bi nk^3\sum_{j=0}^k\bi kj^3.
\end{align*}
and thus
\begin{equation}\label{3,3}S_3^{(3)}(n)=\sum_{k=0}^n\bi nk^3f_k.
\end{equation}

For $a,b,c,d\in\N$ with $a+b+c+d=n$, clearly
$$\bi n{a,b,c,d}=\bi n{a+b}\bi{a+b}a\bi{n-a-b}c.$$
Thus
\begin{align*}\sum_{a,b,c,d\in\N\atop a+b+c+d=n}\bi n{a,b,c,d}^2
&=\sum_{a,b,c,d\in\N\atop a+b+c+d=n}\bi n{a+b}^2\bi{a+b}a^2\bi{n-a-b}c^2
\\&=\sum_{k=0}^n\bi nk^2\sum_{a=0}^k\bi ka^2\sum_{c=0}^{n-k}\bi{n-k}c^2
\\&=\sum_{k=0}^n\bi nk^2\bi{2k}k\bi{2(n-k)}{n-k}
\end{align*}
This shows that
\begin{equation}\label{S-D}S_4^{(2)}(n)=D(n).
\end{equation}

In this paper, we establish the following general congruence modulo primes.

\begin{theorem} \label{Th1.1} Let $l\in\Z^+$ and $m\in\N$.
For any prime $p$ not dividing $l+1$, we have
the congruence
\begin{equation}\label{S-cong}\sum_{n=1}^{p-1}\f{(-1)^{mn}}{n^{m-1}}S_l^{(m)}(n)\eq0\pmod p.
\end{equation}
\end{theorem}

Note that Theorem \ref{Th1.1} with $l=2$ yields the congruence \eqref{f-p}.

In view of \eqref{S-D}, Theorem \ref{Th1.1} with $l=4$ and $m=2$ yields the congruence
\begin{equation}\label{D-cong}\sum_{n=1}^{p-1}\f{D(n)}n\eq0\pmod p.
\end{equation}
for any prime $p\not=5$.

Recall that the $q$-analogue of an integer $n$ is defined by
$$[n]_q:=\f{1-q^n}{1-q}=\begin{cases}\sum_{0\ls k<n}q^k&\t{if}\ n\gs0,
\\-q^n\sum_{0<k<|n|}q^k&\t{if}\ n<0,\end{cases}$$
which tends to $n$ as $q\to1$. For $N\in\Z^+$, the cyclotomic polynomial
$$\Phi_N(q):=\prod_{a=1\atop (a,N)=1}^N\l(q-e^{2\pi i a/N}\r)$$
is irreducible in the ring $\Z[q]$. Note that $[n]_q=(1-q^n)/(1-q)$
is relatively prime to $\Phi_N(q)$ for any positive integer $n<N$.

Let
$$[0]_q!=1\ \ \t{and}\ \ [n]_q!=\prod_{k=1}^n[k]_q\ \ \t{for}\ n=1,2,3,\ldots.$$
For $k_1,\ldots,k_l\in\N$ with $k_1+\cdots+k_l=n$, the $q$-analogue of the multinomial coefficient
$\bi n{k_1,\ldots,k_l}$ is given by
$$\qbinom{n}{k_1,\ldots,k_l}:=\f{[n]_q!}{[k_1]_q!\cdots[k_l]_q!},$$
which is a polynomial in $q$ with integer coefficients.

For any $l\in\Z^+$ and $m,n\in\N$, we define the polynomial
\begin{equation}\label{qS}
S_l^{(m)}(n;q):=\sum_{k_1,\ldots,k_l\in\N\atop k_1+\cdots+k_l=n}q^{\sum_{j=1}^l jk_j}\qbinom{n}{k_1,\ldots,k_l}^m,
\end{equation}
which is a $q$-analogue of the number $S_l^{(m)}(n)$.

Our next theorem is a $q$-analogue of Theorem \ref{Th1.1}.

\begin{theorem}\label{Th-q} Let $l\in\Z^+$ and $m,n\in\N$. For any integer $N>1$, we have
\begin{equation}\label{q-cong}\sum_{0<n<N}\f{(-1)^{mn}q^{-m\bi n2-(l+1)n}}{[n]_q^{m-1}}[l+1]_{q^n}S_l^{(m)}(n;q)\eq0\pmod{\Phi_N(q)}.
\end{equation}
\end{theorem}
\begin{remark} As $\Phi_p(q)=[p]_q$ for any prime $p$, Theorem \ref{Th-q} with $q\to1$
yields Theorem \ref{Th1.1}.
\end{remark}

Our third theorem confirms \cite[Conjecture 79(i)]{S19} posed by the author in 2019.

\begin{theorem} \label{Th1.2}  For any prime $p$, we have the congruence
\begin{equation}\label{D-conj}\sum_{n=1}^{p-1}\f{D(n)}n\eq\l(\f p3\r)\f 25pB_{p-2}\l(\f13\r)\pmod{p^2},
\end{equation}
where $(\f p3)$ is the Legendre symbol, and $B_{p-2}(x)$ is the Bernoulli polynomial of degree $p-2$.
\end{theorem}

We are going to prove Theorems \ref{Th1.1}-\ref{Th-q} in the next section by using derivatives and $q$-derivatives.
The more sophisticated proof of Theorem \ref{Th1.2} will be given in Section 3.
We will pose some new conjectures in Section 4.

\section{Proofs of Theorems \ref{Th1.1}-\ref{Th-q}}
\setcounter{lemma}{0}
\setcounter{theorem}{0}
\setcounter{corollary}{0}
\setcounter{remark}{0}
\setcounter{equation}{0}

\medskip
\noindent{\it Proof of Theorem \ref{Th1.1}}.  Define
$$F(x):=\sum_{k=0}^{p-1}\f{x^k}{(k!)^m}.$$
For $n=0,\ldots,p-1$, clearly $[x^n]F(x)^l$ (the coefficient of $x^n$ in the expansion of $F(x)^l$) coincides with
$$\sum_{k_1,\ldots,k_l\in\N\atop k_1+\cdots+k_l=n}\f1{(k_1!)^m\cdots(k_l!)^m}=\f{S_l^{(m)}(n)}{(n!)^m}.$$
 Note that
$$F'(x)=\sum_{j=1}^{p-1}\f{jx^{j-1}}{j^m\times((j-1)!)^m}=\sum_{k=0}^{p-2}\f{x^k}{(k+1)^{m-1}(k!)^m}.$$
Thus
\begin{align*}[x^{p-1}]F(x)^l F'(x)
&=\sum_{k=0}^{p-1}[x^{p-1-k}]F(x)^l\times [x^k]F'(x)
\\&=\sum_{k=0}^{p-2}\f{S_l^{(m)}(p-1-k)}{((p-1-k)!)^m}\times\f1{(k+1)^{m-1}(k!)^m}
\\&=\sum_{n=1}^{p-1}\f{S_l^{(m)}(n)}{(p-n)^{m-1}(n!(p-1-n)!)^m}
\\&=\f1{((p-1)!)^m}\sum_{n=1}^{p-1}\bi{p-1}n^m\f{S_l^{(m)}(n)}{(p-n)^{m-1}}
\end{align*}
Combining this with the fact $$(F(x)^{l+1})'=(l+1)F(x)^lF'(x),$$ we get
\begin{equation}\label{l+1}\begin{aligned}p\times[x^p]F(x)^{l+1}&=[x^{p-1}](l+1)F(x)^lF'(x)
\\&=\f{l+1}{((p-1)!)^m}\sum_{n=1}^{p-1}\bi{p-1}n^m\f{S_l^{(m)}(n)}{(p-n)^{m-1}}.
\end{aligned}\end{equation}
Note that $(p-1)!\eq-1\pmod p$ by Wilson's theorem. Also, $$\bi{p-1}n=\prod_{k=1}^n\f{p-k}k\eq(-1)^n\pmod p$$
for all $n=1,\ldots,p-1$.
As $p\nmid(l+1)$, we obtain from \eqref{l+1} the desired congruence \eqref{S-cong}. This ends the proof. \qed

Our proof of Theorem \ref{Th-q} is motivated by the above proof of Theorem \ref{Th1.1}.

\medskip
\noindent{\it Proof of Theorem \ref{Th-q}}.
For a polynomial $P(x)=\sum_{r\in\N}c_rx^r$, its Jackson $q$-derivative (cf. \cite[p.\,488]{AAR}) is defined by
$$ \partial_q P(x):=\f{P(x)-P(qx)}{(1-q)x}=\sum_{r\in\Z^+}[r]_q c_r x^{r-1},$$
and hence
\begin{equation}\label{N-1}[x^{N-1}]\partial_q P(x)=[N]_qc_N=[N]_q[x^N]P(x).
\end{equation}

Set
$$F(x)=\sum_{r=0}^{N-1}\f{x^r}{([r]_q!)^m},\ G(x)=\prod_{j=0}^l F(q^jx)
\ \t{and}\ H(x)=\prod_{j=1}^l F(q^jx).$$
Then
$$G(x)=F(x)H(x)\ \ \t{and}\ \ G(qx)=\prod_{j=1}^{l+1}F(q^jx)=F(q^{l+1}x)H(x).$$
Hence
$$\partial_q G(x)=\f{G(x)-G(qx)}{(1-q)x}=H(x)\f{F(x)-F(q^{l+1}x)}{(1-q)x}.$$
Note that
\begin{align*}\f{F(x)-F(q^{l+1}x)}{(1-q)x}&=\sum_{r=1}^{N-1}\f{1-q^{(l+1)r}}{1-q}\cdot
\f{x^{r-1}}{([r]_q!)^m}
\\&=\sum_{0<r<N}\f{[r]_q[l+1]_{q^r}}{([r]_q!)^m}x^{r-1}
\\&=\sum_{0<r<N}\f{[l+1]_{q^r}}{[r]_q^{m-1}([r-1]_q!)^m}
x^{r-1}.
\end{align*}
Combining these with \eqref{N-1}, we deduce that
\begin{align*}[N]_q[x^N]G(x)&=[x^{N-1}]\partial_q G(x)=[x^{N-1}]\f{F(x)-F(q^{l+1}x)}{(1-q)x}H(x)
\\&=\sum_{0<r<N}\f{[l+1]_{q^r}}{[r]_q^{m-1}([r-1]_q!)^m}[x^{N-r}]H(x).
\end{align*}
For any positive integer $n<N$, clearly
$$[x^n]H(x)=\sum_{k_1,\ldots,k_l\in\N\atop k_1+\cdots+k_l=n}\f{q^{\sum_{j=1}^l jk_j}}{([k_1]_q!\cdots[k_l]_q!)^m}=\f{S_l^{(m)}(n;q)}{([n]_q!)^m}.$$
Therefore,
$$[N]_q[x^N]G(x)=\sum_{0<n<N}\f{[l+1]_{q^{N-n}}}{[N-n]_q^{m-1}([N-n-1]_q!)^m}
\cdot\f{S_l^{(m)}(n;q)}{([n]_q!)^m}$$
and hence
\begin{equation}\label{Nq}([N-1]_q!)^m[N]_q[x^N]G(x)
=\sum_{0<n<N}\qbinom{N-1}n^m\f{[l+1]_{q^{N-n}}}{[N-n]_q^{m-1}}S_l^{(m)}(n;q).
\end{equation}

For $0<n<N$, as $\Phi_N(q)$ divides $[N]_q=(1-q^N)/(1-q)$, we have
$$[l+1]_{q^{N-n}}\eq[l+1]_{q^{-n}}=q^{-ln}[l+1]_{q^n}\pmod{\Phi_N(q)},$$
$$[N-n]_q=\f{1-q^{N-n}}{1-q}\eq \f{1-q^{-n}}{1-q}=-q^{-n}[n]_q\pmod{\Phi_N(q)},$$
and
$$\qbinom{N-1}n=\prod_{j=1}^n\f{[N-j]_q}{[j]_q}\eq\prod_{j=1}^n(-q^{-j})=(-1)^nq^{-n(n+1)/2}\pmod{\Phi_N(q)}.$$
So, \eqref{Nq} implies that
$$\sum_{0<n<N}\f{(-1)^{mn}q^{-mn(n+1)/2}q^{-ln}}{(-q^{-n}[n]_q)^{m-1}}[l+1]_{q^n}
S_l^{(m)}(n;q)\eq0\pmod{\Phi_N(q)},$$
which is equivalent to the desired congruence \eqref{q-cong}.
This concludes our proof of Theorem \ref{Th-q}. \qed

\section{Proof of Theorem \ref{Th1.2}}
\setcounter{lemma}{0}
\setcounter{theorem}{0}
\setcounter{corollary}{0}
\setcounter{remark}{0}
\setcounter{equation}{0}

\begin{lemma}\label{Lem3.1} For any prime $p>5$, we have 
\begin{equation}\label{Ber}\sum_{n=1}^{p-1}\f{\bi{2n}n}{n^2}\eq\f12\l(\f p3\r)B_{p-2}\l(\f13\r)\pmod p.
\end{equation}
\end{lemma}

\begin{remark} This lemma is known (cf. \cite{MT}). One may prove it by starting from the observation
\begin{align}\sum_{n=1}^{p-1}\f{\bi{2n}n}{n^2}&\eq\sum_{n=1}^{(p-1)/2}\f{\bi{2n}n}{n^2}
=\sum_{n=1}^{(p-1)/2}\bi{-1/2}n\f{(-4)^n}{n^2}
\\&\eq\sum_{n=1}^{(p-1)/2}\bi{(p-1)/2}n\f{(-4)^n}{n^2}\pmod p
\end{align}
and using 
$$\sum_{k=1}^{\lfloor p/3\rfloor}\f1{k^2}\eq\f15\sum_{k=1}^{\lfloor p/6\rfloor}\f1{k^2}
\eq\f12\l(\f p3\r)B_{p-2}\l(\f13\r)\pmod p.$$
\end{remark}

\medskip
\noindent{\it Proof of Theorem \ref{Th1.2}}. It is easy to verify \eqref{D-conj} for $p=2,3,5$ directly. Below we assume $p>5$.

Recall that the harmonic numbers are given by
$$H_n:=\sum_{0<k\ls n}\f1k\ \ (n=0,1,2,\ldots).$$
By \eqref{S-D} we have
\begin{align*}\sum_{n=1}^{p-1}\f{D(n)}n&=\sum_{n=1}^{p-1}\f1{p-n}\sum_{a,b,c,d\in\N
\atop a+b+c+d=p-n}\bi{p-n}{a,b,c,d}^2
\\&=\sum_{n=1}^{p-1}\f{p^2}{(p-n)\bi pn^2}\sum_{a,b,c,d\in\N\atop a+b+c+d=p-n}\f{((p-1)!)^2}{(a!)^2(b!)^2(c!)^2(d!)^2(n!)^2}
\end{align*}
For $n\in\{1,\ldots,p-1\}$, clearly
\begin{align*}\f{p^2}{(p-n)\bi pn^2}&=\f{n^2}{(p-n)\bi{p-1}{n-1}^2}
=\f{(p+n)n^2}{(p^2-n^2)\prod_{0<k<n}(p/k-1)^2}
\\&\eq-\f{p+n}{\prod_{0<k<n}(1-2p/k)}\eq-(p+n)\prod_{0<k<n}\l(1+\f{2p}k\r)
\\&\eq-(p+n)(1+2pH_{n-1})\eq-n-p(2nH_n-1)\pmod{p^2}.
\end{align*}
Set $$A=\{{\bf a}=(a_1,\ldots,a_5)\in\N^5:\ a_1+\cdots+a_5=p\ \&\ |\{1\ls r\ls 5:\ a_r\not=0\}|\gs2\},$$
and define
$$Q({\bf a})=Q(a_1,\ldots,a_5)=\f{((p-1)!)^2}{(a_1!)^2(a_2!)^2(a_3!)^2(a_4!)^2(a_5!)^2}$$
for ${\bf a}\in A$.
 Let $\da_0=0$ and $\da_n=1$ for $n=1,\ldots,p-1$. By the above,
\begin{align*}5\sum_{n=1}^{p-1}\f{D(n)}n&\eq-5\sum_{n=1}^{p-1}(n+p(2nH_n-\da_n))\sum_{a,b,c,d\in\N\atop a+b+c+d=p-n}Q(a,b,c,d,n)
\\&\eq-\sum_{{\bf a}\in A}\sum_{r=1}^5(a_r+p(2a_rH_{a_r}-\da_{a_r}))Q(a_1,\ldots,a_5)
\pmod {p^2}
\end{align*}
and hence
\begin{equation}\label{pq}\f5p\sum_{n=1}^{p-1}\f{D(n)}n
\eq-\sum_{{\bf a}\in A}\f{1+2\sum_{r=1}^5a_rH_{a_r}
-\sum_{r=1}^5 \da_{a_r}}{(a_1!)^2(a_2!)^2(a_3!)^2(a_4!)^2(a_5!)^2}
\pmod{p}.
\end{equation}

Set
$$F(x)=\sum_{k=0}^{p-1}\f{x^k}{(k!)^2}\ \ \t{and}\ \ G(x)=\sum_{k=0}^{p-1}\f{H_kx^k}{(k!)^2}.$$
Then
$$\sum_{{\bf a}\in A}\f1{(a_1!)^2\cdots(a_5!)^2}=[x^p]F(x)^5$$
and
$$\sum_{{\bf a}\in A}\f{\sum_{r=1}^5\da_{a_r}}{(a_1!)^2\cdots(a_5!)^2}
=5[x^p]F(x)^5-5[x^p]F(x)^4.$$
Also,
$$\sum_{{\bf a}\in A}\f{\sum_{r=1}^5 a_rH_{a_r}}{(a_1!)^2\cdots(a_5!)^2}=5[x^{p-1}]F(x)^4G'(x).$$
In view of these, \eqref{pq} has the equivalent form
\begin{equation}\label{equiv}-\f 5p\sum_{n=1}^{p-1}\f{D(n)}n
\eq5[x^p]F(x)^4-4[x^p]F(x)^5+10[x^{p-1}]F(x)^4G'(x)\pmod p.
\end{equation}
Observe that
\begin{align*}D(p)&=\sum_{a,b,c,d\in\N\atop a+b+c+d=p}\f{(p!)^2}{(a!)^2(b!)^2(c!)^2(d!)^2}
\\&=4+\sum_{a,b,c,d\in\{0,\ldots,p-1\}\atop a+b+c+d=p}\f{p^2((p-1)!)^2}{(a!)^2(b!)^2(c!)^2(d!)^2}
\\&=4+p^2((p-1)!)^2[x^p]F(x)^4\eq 4+p^2[x^p]F(x)^4\pmod{p^3}.
\end{align*}
On the other hand,
\begin{align*}D(p)&=\sum_{k=0}^p\bi pk^2\bi{2k}k\bi{2(p-k)}{p-k}
\\&=2\bi{2p}p+\sum_{k=1}^{p-1}\f {p^2}{k^2}\bi{p-1}{k-1}^2\bi{2k}k\bi{2(p-k)}{p-k}
\\&\eq4\bi{2p-1}{p-1}\eq4\pmod{p^3}
\end{align*}
since $\bi{2k}k\bi{2(p-k)}{p-k}\eq0\pmod p$ for all $k=1,\ldots,p-1$, and $\bi{2p-1}{p-1}\eq1\pmod {p^3}$ by Wolstholme's theorem (cf. \cite{W}). Thus $[x^p]F(x)^4\eq0\pmod p$ and hence
we simplify \eqref{equiv} as
\begin{equation}\label{5p}\f 5p\sum_{n=1}^{p-1}\f{D(n)}n
\eq4[x^p]F(x)^5-10[x^{p-1}]F(x)^4G'(x)\pmod p.
\end{equation}

Clearly,
\begin{align*}(xF'(x))'&=F'(x)+xF''(x)=\sum_{k=1}^{p-1}\f{(k+k(k-1))x^{k-1}}{k^2((k-1)!)^2}
\\&=\sum_{k=1}^{p-1}\f{x^{k-1}}{((k-1)!)^2}=F(x)-\f{x^{p-1}}{((p-1)!)^2}
\\&\eq F(x)-x^{p-1}\pmod p
\end{align*}
and
\begin{align*}(xG'(x))'&=G'(x)+xG''(x)=\sum_{k=1}^{p-1}\f{(k+k(k-1))H_kx^{k-1}}{k^2((k-1)!)^2}
\\&=\sum_{k=1}^{p-1}\f{(H_{k-1}+1/k)x^{k-1}}{((k-1)!)^2}=G(x)-\f{H_{p-1}x^{p-1}}{((p-1)!)^2}
+\sum_{k=1}^{p-1}\f{kx^{k-1}}{(k!)^2}
\\&\eq G(x)+F'(x)\pmod p.
\end{align*}
For the polynomial
$$W(x):=x(F(x)G'(x)-F'(x)G(x))=F(x)(xG'(x))-(xF'(x))G(x),$$
by the above we have
\begin{align*}
W'(x)&=F(x)(xG'(x))'+F'(x)xG'(x)-xF'(x)G'(x)-(xF'(x))'G(x)
\\&\eq F(x)(F'(x)+G(x))-(F(x)-x^{p-1})G(x)
\\&\eq F(x)F'(x)+x^{p-1}G(x)\pmod p
\end{align*}
and hence
\begin{equation}\label{W}\l(W(x)-\f{F(x)^2-1}2\r)'\eq x^{p-1}G(x)\pmod p.
\end{equation}
Since
$$[x^p]F(x)^2=\sum_{k=1}^{p-1}\f1{(k!)^2((p-k)!)^2}
=\sum_{k=1}^{p-1}\f{\bi {p-1}k^2}{(p-k)^2((p-1)!)^2},$$
we have
\begin{equation}\label{F2}
[x^p]F(x)^2\eq\sum_{k=1}^{p-1}\f1{k^2}=2\sum_{k=1}^{(p-1)/2}\l(\f1{k^2}+\f1{(p-k)^2}\r)\eq0\pmod p
\end{equation}
and hence
\begin{align*}&\ [x^p]\l(W(x)-\f{F(x)^2-1}2\r)
\\\eq&\ [x^p]W(x)=[x^{p-1}](F(x)G'(x)-F'(x)G(x))
\\\eq&\ \sum_{k=1}^{p-1}\l(\f1{(k!)^2}\times\f{(p-k)H_{p-k}}{((p-k)!)^2}
-\f{k}{(k!)^2}\times\f{H_{p-k}}{((p-k)!)^2}
\r)
\\\eq&\ \sum_{k=1}^{p-1}\f{\bi{p-1}k^2}{((p-1)!)^2}\l(\f{H_{p-k}}{p-k}-\f{kH_{p-k}}{(p-k)^2}\r)
\\\eq&\ 2\sum_{k=1}^{p-1}\f{H_{p-k}}{p-k}=2\sum_{k=1}^{p-1}\f{H_k}k
\eq0\pmod p
\end{align*}
with the aid of \cite[Lemma 2.3]{S12} and its proof.
Combining this with \eqref{W} we find that
$$W(x)-\f{F(x)^2-1}2\eq x^{p+1}P_p(x)\pmod p$$
for some $P_p(x)\in\Z_p[x]$, where $\Z_p$ is the ring of $p$-adic integers.
Thus,
$$[x^p]W(x)F(x)^3\eq[x^p]\f{F(x)^2-1}2F(x)^3=\f12[x^p]F(x)^5-\f12[x^p]F(x)^3\pmod p.$$
On the other hand,
\begin{align*}[x^p]W(x)F(x)^3&=[x^{p-1}](F(x)G'(x)-F'(x)G(x))F(x)^3
\\&=[x^{p-1}]F(x)^4G'(x)-\f14[x^{p-1}](F(x)^4)'G(x)
\\&=\f54[x^{p-1}]F(x)^4G'(x)-\f14[x^{p-1}](F(x)^4G(x))'
\\&\eq\f54[x^{p-1}]F(x)^4G'(x)\pmod p.
\end{align*}
Therefore,
\begin{equation}\label{52}\f52[x^{p-1}]F(x)^4G'(x)\eq[x^p]F(x)^5-[x^p]F(x)^3\pmod p.
\end{equation}

Combining \eqref{52} with \eqref{5p}, we deduce that
\begin{equation}\label{cube}\f5p\sum_{n=1}^{p-1}\f{D(n)}n\eq4[x^p]F(x)^3\pmod p.
\end{equation}
Observe that
$$[x^p]F(x)^3=[x^p]F(x)^2\sum_{n=1}^p\f{x^{p-n}}{((p-n)!)^2}
=\sum_{n=1}^{p-1}\f{[x^n]F(x)^2}{((p-n)!)^2}+[x^p]F(x)^2$$
and
$$[x^n]F(x)^2=\sum_{k=0}^n\f1{(k!)^2((n-k)!)^2}=\f{\sum_{k=0}^n\bi nk^2}{(n!)^2}=\f{\bi{2n}n}{(n!)^2}$$
for all $n=1,\ldots,p-1$. Thus, with the aid of \eqref{F2}, we have
\begin{align*}[x^p]F(x)^3&\eq\sum_{n=1}^{p-1}\f{\bi{2n}n}{(n!)^2((p-n)!)^2}
=\sum_{n=1}^{p-1}\f{\bi{p-1}n^2\bi{2n}n}{(p-n)^2((p-1)!)^2}
\\&\eq\sum_{n=1}^{p-1}\f{\bi{2n}n}{n^2}\pmod p.
\end{align*}
Thus, in view of \eqref{Ber}, we get
$$[x^p]F(x)^3\eq\f12\l(\f p3\r)B_{p-2}\l(\f13\r)\pmod p.$$
Combining this with \eqref{cube}, we obtain the desired congruence \eqref{D-conj}.
This completes our proof of Theorem \ref{Th1.2}. \qed

\section{Some further conjectures}
\setcounter{lemma}{0}
\setcounter{theorem}{0}
\setcounter{corollary}{0}
\setcounter{conjecture}{0}
\setcounter{remark}{0}
\setcounter{equation}{0}

Our first conjecture involves the Domb numbers and the Lucas sequences $(u_n)_{n\ge0}$ and $(v_n)_{n\ge0}$, where
$$u_0=0,\ u_1=1,\ \text{and}\ u_{n+1}=11u_n-u_{n-1}\ \text{for}\ n=1,2,3,\ldots,$$
and
$$v_0=2,\ v_1=11,\ \text{and}\ v_{n+1}=11v_n-v_{n-1}\ \text{for}\ n=1,2,3,\ldots.$$

\begin{conjecture} Let $p$ be an odd prime.

{\rm (i)} If one of the Legendre symbols $(\frac p3)$ and $(\frac p{13})$ is $1$, then
$$\sum_{k=0}^{p-1}D(k)u_k\equiv0\pmod{p^2}.$$
Moreover, when $(\frac p3)=(\frac p{13})=1$, we have
$$\sum_{k=0}^{p-1}D(k)u_k\equiv0\pmod{p^3}.$$

{\rm (ii)} If the Jacobi symbol $(\frac p{39})$ is $-1$, then
$$\sum_{k=0}^{p-1}D(k)v_k\equiv0\pmod{p^2}.$$
\end{conjecture}

In the spirit of \cite{S13b}, we make the following conjecture based on our computation.

\begin{conjecture} The sequence $(S_3^{(3)}(n+1)/S_3^{(3)}(n))_{n\gs0}$
is strictly increasing to the limit $27$, and the sequence 
$$\l(\!\!\root{n+1}\of{S_3^{(3)}(n+1)}\bigg/\root n\of{S_3^{(3)}(n)}\r)_{n\gs1}$$
is strictly decreasing to the limit $1$.
\end{conjecture}


\begin{thebibliography}{99}

\bibitem{AAR} G. E. Andrews, R. Askey and R. Roy, Special Functions,
Encyclopedia of Mathematics and its Applications, vol. 71,
Cambridge Univ. Press, Cambridge, 1999.

\bibitem{Cam} J. M. Campbell, {\it On modular forms and Domb numbers},
Integral Transforms Spec. Funct. {\bf 36} (2025), 439--448.

\bibitem{CCL} H. H. Chan, S. H. Chan and Z. Liu,
{Domb's numbers and Ramanujan-Sato type series for $1/\pi$},
Adv. Math. {\bf 186} (2004), 396--410.

\bibitem{F} J. Franel, {\it On a question of Laisant},
 L'Interm\'ediaire des Math\'ematiciens, {\bf 1} (1894), 45--47.

\bibitem{GKP} R. L. Graham, D. E. Knuth and O. Patashnik,
  Concrete Mathematics, 2nd ed., Addison-Wesley, New York, 1994.

\bibitem {GMP} V.J.W. Guo, G.-S. Mao and H. Pan, {\it Proof of a conjecture involving Sun polynomials}, J. Difference Equ. Appl. {\bf 22} (2016), 1184--1197.

\bibitem{Liu} J.-C. Liu, {\it Supercongruences for sums involving Domb numbers},
Bull. Sci. Math. {\bf 169} (2021), Article 102992.

 \bibitem{MT} S. Mattarei and R. Tautaso,{\it Congruences for central binomial sums and finite polylogarithms}, J. Number Theory {\bf 133} (2013), 131--157.

\bibitem{Rogers}
M. D. Rogers,
{\it New ${}_5F_4$ hypergeometric transformations, three-variable Mahler measures, and formulas for $1/\pi$},
 Ramanujan J. {\bf 18} (2009), 327--340.

\bibitem{OEIS} N.J.A. Sloane, Sequence A002895 (Domb numbers) at OEIS (On-Line Encyclopedia of Integer Sequences), https://oeis.org/A002895.

\bibitem{AAM} Z.-W. Sun, {\it Congruences for Franel numbers}, Adv. Appl. Math. {\bf 51} (2013), 524--535.

\bibitem{S12} Z.-W. Sun, {\it Arithmetic theory of harmonic numbers}, Proc. Amer. Math. Soc.
{\bf 140} (2012), 415--428.

\bibitem{S13} Z.-W. Sun,  {\it Conjectures and results on $x^2$ mod $p^2$ with $4p=x^2 + dy^2$}, in: Number Theory and Related Area (eds., Y. Ouyang, C. Xing, F. Xu and P. Zhang), Adv. Lect. Math. 27, Higher Education Press and International Press, Beijing-Boston, 2013, pp. 149--197.

\bibitem{S13b} Z.-W. Sun, {\it Conjectures involving arithmetical sequences}, in: Number Theory: Arithmetic in Shangri-La (eds., S. Kanemitsu,
H. Li and J. Liu), Proc. 6th China-Japan Seminar (Shanghai,
August 15-17, 2011), World Sci., Singapore, 2013, pp.
244--258.

\bibitem{S16} Z.-W. Sun, {\it Congruences involving $g_n(x)=\sum_{k=0}^n\bi nk^2\bi{2k}kx^k$},
Ramanujan J. {\bf 40} (2016), 511--533.

\bibitem{S19} Z.-W. Sun, {\it Open conjectures on congruences}, Nanjing Univ. J. Math. Biquarterly
{\bf 36} (2019), no.\,1, 1--99.

\bibitem{S20} Z.-W. Sun, {\it New type series for powers of $\pi$}, J. Comb. Number Theory
{\bf 12} (2020), no.\,3, 157--208.

\bibitem{W} J. Wolstenholme, {\it On certain properties of prime numbers}, Quart. J. Appl. Math.
{\bf 5} (1862), 35--39.

\end{thebibliography}
\end{document}